\documentclass[12pt]{article}

\usepackage[paper=a4paper,left=30mm,right=30mm,top=30mm,bottom=30mm]{geometry}
\usepackage{algorithm} 
\usepackage{algorithmic} 
\usepackage{amsmath} 
\usepackage{amssymb} 
\usepackage{amsthm} 
\usepackage[shortlabels]{enumitem}
    \setlist{nosep}
\usepackage[utf8]{inputenc}
\usepackage[T1]{fontenc}
\usepackage{mathptmx}
\usepackage[scaled=0.8]{beramono}
\usepackage{pgfplots}
\usepackage{tikz}
    \pgfplotsset{compat=newest}
    \usetikzlibrary{plotmarks}
    \usetikzlibrary{calc}
    \usetikzlibrary{positioning}
    \usetikzlibrary{arrows.meta}
    \usetikzlibrary{matrix}
    \usepgfplotslibrary{patchplots}
\usepackage[labelsep=period]{caption}
\usepackage{multicol}
\usepackage{xspace}
\usepackage{xcolor} 
\usepackage[all]{nowidow}
\usepackage{cite}
\usepackage{hyperref}
\usepackage[nameinlink]{cleveref} 

\let\oldeqref\eqref
\renewcommand{\eqref}[1]{\oldeqref{eq:#1}}

\mathchardef\ordinarycolon\mathcode`\:
\mathcode`\:=\string"8000
\begingroup \catcode`\:=\active
  \gdef:{\mathrel{\mathop\ordinarycolon}}
\endgroup

\DeclareMathAlphabet{\altmathcal}{OMS}{cmsy}{m}{n}
\renewcommand{\mathcal}{\altmathcal}

\newcommand{\norm}[1]{\Vert #1 \Vert}

\newcommand{\RR}{\mathbb{R}}

\newcommand{\MM}{\mathcal{M}}

\newcommand{\und}{~~ \text{and} ~~}

\newcommand{\Uad}{\mathcal{U}_{\mathrm{ad}}}
\newcommand{\inv}{^{-1}}
\newcommand{\tp}{^{\mathsf{T}}}

\renewcommand{\tilde}{\widetilde}

\DeclareMathOperator*{\ran}{ran}

\DeclareMathOperator*{\dist}{dist}
\DeclareMathOperator*{\unpp}{unpp}

\newcommand{\ddt}{\frac{\mathrm{d}}{\mathrm{d}t}}
\newcommand{\ddx}{\frac{\mathrm{d}}{\mathrm{d}x}}
\newcommand{\dt}{~\mathrm{d}t}

\newcommand{\pH}{\textsf{pH}\xspace}
\newcommand{\portHamiltonian}{port-Ha\-mil\-to\-ni\-an\xspace}

\renewcommand{\vec}[1]{\begin{bmatrix} #1 \end{bmatrix}}

\pgfplotsset{
  custom y log style/.style={
      yticklabel={
        \pgfkeys{/pgf/fpu=true}
        \pgfmathparse{exp(\tick)}%
        \pgfmathprintnumber[fixed,fixed zerofill, precision=2]{\pgfmathresult}
        \pgfkeys{/pgf/fpu=false}
      }
  }
}

\allowdisplaybreaks

\definecolor{mylinkcolor}{rgb}{0.031,0.2,0.369}
\definecolor{myemphcolor}{rgb}{0.2,0.2,0.2}

\hypersetup{
        colorlinks=false,
        pdfborder={0 0 0},
        pdfborderstyle={/S/U/W 0},
        }

\title{Manifold turnpikes of nonlinear \portHamiltonian descriptor systems under minimal energy supply}

\author{Attila Karsai\thanks{Institute of Mathematics, Technische Universität Berlin, Germany. \texttt{karsai@math.tu-berlin.de}}}

\date{\today}

\theoremstyle{plain}
\newtheorem{theorem}{Theorem}

\newtheorem{definition}[theorem]{Definition}
\newtheorem{example}[theorem]{Example}
\newtheorem{lemma}[theorem]{Lemma}
\newtheorem{proposition}[theorem]{Proposition}
\newtheorem{remark}[theorem]{Remark}
\newtheorem{assumption}[theorem]{Assumption}

\allowdisplaybreaks 

\begin{document}

\maketitle

\begin{abstract}
    Turnpike phenomena of nonlinear \portHamiltonian descriptor systems under minimal energy supply are studied.
    Under assumptions on the smoothness of the system nonlinearities, it is shown that the optimal control problem is dissipative with respect to a manifold.
    Then, under controllability assumptions, it is shown that the optimal control problem exhibits a manifold turnpike property.
\end{abstract}


\section{Introduction}

This paper is concerned with \emph{turnpike phenomena}.
These phenomena were first noticed in the context of economics~\cite{DorSS87,McK63} and have since been observed in many different situations, see, e.g.,~\cite{CarHL91,Zas06,TreZ18} and the references therein.
Usually, turnpike phenomena are studied in optimal control problems, where the goal is the minimization of a cost functional $C(u)$.
Here, the function $u$ acts as the control of a system of interest. 
In many cases, it can be observed that an optimal control~$u^*$ will, for a majority of the time horizon, steer the associated state trajectory~$x^*$ to a point~\cite{PorZ13,PorZ16,TreZ15}, a set~\cite{Zas06,TreZ18,RapC04} or, as in our case, a manifold~\cite{FauFO22}.
In other words, optimal solutions depend mainly on the underlying system and the optimization objective and are more or less independent on the choice of the time horizon and other data, such as initial or final values.
The behaviour is reminiscent of an observation from daily life:
when traveling a long distance by car, it is usually faster to take a detour via a turnpike than to choose a more direct way on slower streets.
Also, the chosen turnpike usually does not heavily depend on the start and end of the route.
If one would start the journey a few blocks away, then the fastest path would remain more or less the same.

Here, we consider a special class of systems called \emph{\portHamiltonian} (\pH) systems.
Parts of the origins of \portHamiltonian systems date back to the late 1950s~\cite{Pay61}, and the interested reader is referred to~\cite{VanJ14} for an overview on the origins of this system class.
Despite their long history, they continue to be the focus of active research~\cite{MosLYY,LamH22,SchM22,BreHU22,GerH22,MehU22}.
Arguably, the key feature of \pH systems is their modeling perspective: they focus on taking energy as the \emph{lingua franca} between subsystems.
As a consequence, the class of \pH systems is a promising class for modeling real world processes~\cite{MehU22}.
Benefits of \portHamiltonian models include inherent stability and passivity, the invariance under Galerkin projection and congruence transformation, and the possibility to interconnect multiple~\pH systems in a structure-preserving manner.

When \pH systems are considered in an optimal control setting, the objective of minimizing the supplied energy is quite natural.
This results in a cost term of the form $\int_0^T y(t)\tp u(t) \dt$, where $y$ is a collocated observation of the system, and renders the corresponding optimal control problem singular.
In~\cite{SchPFWM21,FauMP22}, the authors have considered this objective for linear time invariant \portHamiltonian (descriptor) systems.
They have shown that the optimal control problem has a measure turnpike property with respect to the dissipative part of the state space, given by the kernel of the matrix corresponding to the non-conservative system dynamics.
The infinite-dimensional linear case was discussed in~\cite{PhiSF21}.
In this paper, we are concerned with the finite-dimensional nonlinear case.
We show that, under smoothness assumptions on the nonlinearities and controllability assumptions on the system, nonlinear \pH descriptor systems admit a turnpike phenomenon with respect to a submanifold of $\RR^n$.
This submanifold corresponds, as in the linear case, to the energy dissipating part of the state space.

The structure of this paper is as follows.
In \Cref{sec:preliminaries}, we recall the definition of \portHamiltonian systems and precisely state the optimal control problem that is considered.
After that, a short repetition of results on submanifolds of $\RR^n$ follows in \Cref{sec:manifolds}.
In \Cref{sec:dissip}, we define manifold dissipativity and manifold turnpikes following~\cite{FauFO22} and recall that, under weak assumptions, manifold dissipativity implies a manifold turnpike property.
\Cref{sec:app-ph} contains the main results of this work, where the previously established results are applied to finite-dimensional nonlinear \portHamiltonian descriptor systems.
The theoretical results are then illustrated by a numerical example in \Cref{sec:num}.
Finally, in \Cref{sec:conclusion} a conclusion is drawn and an outlook on future research is given.

\subsubsection*{Notation}
For a set $Z \subseteq \RR^n$ we define $Z^\circ$ as the interior of $Z$.
We denote the Euclidean norm by $\norm{\cdot}$ and define the distance of a point $x\in\RR^n$ to the set $\MM\subseteq\RR^n$ as $\dist(x,\MM) := \inf_{p \in \MM} \norm{\hbox{$x-p$}}$.
We denote the set of all $k$-times continuously differentiable functions from $U$ to $V$ by $C^k(U,V)$ and define $C(U,V) := C^0(U,V)$.
When the spaces $U$ and $V$ are clear from context, we say that $f\in C^k(U,V)$ is of class $C^k$.
The derivative of a function~$f$ at point $x$ is denoted by $Df_x$.
Further, for a matrix $A\in\RR^{n,n}$ we write $A \succeq 0$ if $x\tp A x \geq 0$ for all $x\in \RR^{n}$, and $A \succ 0$ if $x\tp A x > 0$ for all $x\in\RR^{n}\setminus\{0\}$, where $\cdot\tp$ denotes the transpose.
The kernel and range of the matrix~$A$ are denoted by $\ker(A)$ and $\ran(A)$, respectively.
The non-negative square-root of a positive semidefinite matrix $A$ is denoted by $A^{1/2}$.
Often, we surpress the time dependency of functions and write $z$ instead of~$z(t)$.

\section{Preliminaries and problem setting}\label{sec:preliminaries}

Following the definiton of~\cite{MehU22}, we consider nonlinear \portHamiltonian descriptor systems of the form
\begin{equation}\label{eq:nlph}
    \begin{aligned}
        E(x) \dot{x} = &~ (J(x) - R(x))\eta(x) + B(x) u, \\
        y = &~ B(x)\tp \eta(x).
    \end{aligned}
    \tag{\pH}
\end{equation}
Here, $x$, $u$ and $y$ are the state, input and output of the system, respectively.
We restrict our analysis to \pH systems without feedthrough but note that the discussion can be extended to systems with a feedthrough term as introduced in~\cite{MehU22}.
We consider the state space $\mathcal{X} = \RR^n$ and a set of admissible controls $\Uad$ and require that
\begin{equation*}
    E,J, R \in C(\mathcal{X}, \RR^{n,n}),~~ \eta \in C(\mathcal{X},\RR^n) \und B \in C(\mathcal{X},\RR^{n,m}).
\end{equation*}
Further, the functions $J$ and $R$ have to satisfy $J(x) = - J(x)\tp$ and $R(x) = R(x)\tp \succeq 0$ for all $x \in \mathcal{X}$.
We assume that the system~\eqref{nlph} is associated with a Hamiltonian $\mathcal{H} \in C^1(\mathcal{X},\RR)$ which is bounded from below along any solution of~\eqref{nlph} and satisfies
\begin{equation*}
    \ddx \mathcal{H}(x) = E(x)\tp \eta(x)
\end{equation*}
for each $x\in\mathcal{X}$.
Without loss of generality, as in~\cite{MehU22} we may assume that $\mathcal{H}$ is nonnegative along any solution of~\eqref{nlph}.

\begin{remark}\label{rem:ph-time}
    We only consider \pH systems which do not explicitly depend on time.
    The definition given here can be generalized to include explicit time dependence, but these systems can easily be made autonomous~\cite[Remark 4.2]{MehU22}.
\end{remark}

When it comes to the optimal control of \portHamiltonian systems, the cost functional should take into account that \pH systems stem from using energy as the lingua franca.
Hence, choosing the supplied energy as the optimization objective is quite natural.
For this, usually the impedance supply $y\tp u$ is considered, which is related to the scattering supply $\norm{u}^2 - \norm{y}^2$ via the Cayley transform~\cite{Sta02}.
We focus on the former and thus consider the optimal control problem
\begin{equation}\label{eq:ocp}
    \left.
    \begin{aligned}
        \min_{u \in \Uad}~C_{\pH,T}(u) := & \int_{0}^{T} y\tp u \dt \\
        \text{subject to the dynamics~\eqref{nlph} and}\hspace{-2.1cm} \\
        x(0) = x_0, ~~ x(T) = x_T. \hspace{-1.3cm}\\
    \end{aligned}
    \tag{$\text{\pH OCP}_T$}
    \hspace{.5cm}
    \right\}
\end{equation}
Here and in the following, we assume $x_0, x_T \in K$, where $K \subseteq \RR^n$ is a compact set.
Further, as we are interested in the properties of optimal solutions to~\eqref{ocp}, throughout the paper we assume that an optimal solution~$u^*$ and a corresponding trajectory $x^*$ exist.
This assumption is quite restrictive and may be violated.
The existence and uniqueness of optimal controls for nonlinear differential-algebraic equations is studied in, e.g.,~\cite{KunM08}.

It can be shown~\cite{MehU22} that the \emph{power-balance} equation
\begin{equation}\label{eq:powerbalance}
    \ddt \mathcal{H}(x) = -\eta(x)\tp R(x) \eta(x) + y\tp u
\end{equation}
holds along any solution $x$ of~\eqref{nlph}.
This allows us to rewrite the cost functional $C_{\pH,T}(u)$ as
\begin{equation*}
    C_{\pH,T}(u) = \mathcal{H}(x_T) - \mathcal{H}(x_0) + \int_{0}^{T}\norm{R(x)^{1/2} \eta(x)}^2 \dt.
\end{equation*}
This equation is called the \emph{energy-balance} equation, and we can interpret each of the terms physically~\cite{OrtVM02}.
The term $\mathcal{H}(x_T) - \mathcal{H}(x_0)$ measures the conserved energy, while the integral term corresponds to the dissipated energy.
By rearranging and plugging in the definiton of $C_{\pH,T}$, we see that
\begin{equation*}
    \mathcal{H}(x_T) - \mathcal{H}(x_0) = \int_{0}^{T}y\tp u - \norm{R(x)^{1/2}\eta(x)}^2 \dt.
\end{equation*}
Note that this implies dissipativity in the sense of Willems~\cite{Wil72a}, and as a consequence shows the aforementioned passivity of \pH systems.
We will use both of these equations in \Cref{sec:app-ph}.

\section{Submanifolds of $\RR^n$ and the orthogonal projection} \label{sec:manifolds}

This section repeats mostly well-known results regarding submanifolds of $\RR^n$, with a focus on manifolds defined as the zero locus of some smooth function.
The main result of the section is \Cref{lem:dist}, which provides an upper bound for the distance of a point to such a manifold.

We begin with recalling the classical defintion of submanifolds of $\RR^n$.
Following~\cite{DudH94,LeoS21}, we distinguish manifolds whose tangent spaces locally satisfy a Lipschitz condition.

\begin{definition}[submanifolds of $\RR^n$]\hphantom{abc}

    \begin{itemize}

    \item
    Let $\MM$ be a subset of $\RR^n$.
    We call $\MM$ an \emph{$s$-dimensional $C^k$ manifold} if
    for each $p\in \MM$ there exists an open neighborhood $U$ of $p$ and a $C^k$ diffeomorphism $\phi: U \to \phi(U)\subseteq \RR^n$ such that
    \begin{equation*}
        \MM \cap U = \{ x\in U ~|~ \phi_{s+1}(x) = \dots = \phi_n(x) =0 \}.
    \end{equation*}
    The function $\phi$ is called a \emph{local coordinate system of $\MM$ at $p$}.

    \item
    Let $\MM \subseteq \RR^n$ be an $s$-dimensional manifold, $p\in \MM$ and let $\phi: U \to \RR^n$ be a local coordinate system of $\MM$ at $p$.
    We define the \emph{tangent space at $p$ relative to $\MM$} as
    \begin{equation*}
        T_p\MM := D\phi_{\phi(p)}\inv(\left\{ y\in\RR^n ~\middle|~ y_{s+1} = \dots = y_n = 0 \right\}).
    \end{equation*}
    The space $N_p\MM := T_p\MM^\perp$ is called the \emph{normal space at $p$ relative to $\MM$}.

    \item
    We call $\MM\subseteq \RR^n$ an \emph{$s$-dimensional $C^{k,1}$ manifold} if $\MM$ is an $s$-di\-men\-sion\-al $C^k$ manifold and for all $p \in \MM$ there exists a set $V\subseteq \MM$ that is open relative to $\MM$ and a positive constant $L>0$ such that $p \in V$ and for all ${\tilde{p}} \in V$ it holds that
    \begin{equation*}
        d_{\mathrm{H}}(T_p\MM, T_{\tilde{p}}\MM) \leq L \norm{p - {\tilde{p}}}.
    \end{equation*}
    Here, $d_{\mathrm{H}}$ denotes the \emph{Hausdorff distance} defined by
    \begin{equation*}
        d_{\mathrm{H}}(T_1, T_2) := \sup \left\{ \inf \left\{ \norm{t_2-t_1} ~\middle|~  t_2 \in T_2 \cap S \right\} ~\middle|~ t_1 \in T_1 \cap S \right\},
    \end{equation*}
    where $S := \left\{ z \in \RR^n ~\middle|~ \norm{z} = 1 \right\}$ is the unit sphere.
    \end{itemize}
\end{definition}

Next, we recall the definition of the orthogonal projection on a manifold from~\cite{LeoS21}.
For this, consider a manifold $\MM\subseteq \RR^n$ and define the set of points with the \emph{unique nearest points property} as
\begin{equation*}
    \unpp(\MM) := \left\{ x\in\RR^n ~\middle|~ \text{there exists a unique}~\xi\in\MM ~\text{with}~ \dist(x,\MM) = \norm{x-\xi} \right\}.
\end{equation*}
Clearly, for each $x\in\unpp(\MM)$ there exists a unique $p(x)\in\MM$ with the property 
\begin{equation*}
    \norm{x-p(x)} = \dist(x,\MM) = \inf_{p \in \MM} \norm{x-p}.
\end{equation*}

\begin{definition}[orthogonal projection on manifold,~\cite{LeoS21}]
    Let $\MM\subseteq \RR^n$ be a manifold.
    The function $p: \unpp(\MM) \to \MM, ~ x \mapsto p(x)$ is called the \emph{orthogonal projection on $\MM$}.
\end{definition}

Often, we write $p_x$ instead of $p(x)$.
The maximal open set on which the orthogonal projection is defined plays a special role in~\cite{LeoS21} and also in our setting.
We will refer to this set as $\mathcal{E}(\MM) := \unpp(\MM)^\circ$.

The next proposition collects selected results on submanifolds of $\RR^n$ defined in a particular manner and will be useful in \Cref{sec:app-ph}, where it will allow us to study the optimal control problem~\eqref{ocp}.

\begin{proposition}\label{prop:manifolds}
    Suppose $f:\RR^n \to \RR^n$ is of class $C^2$, assume that
    \begin{equation*}
        \MM := \left\{ x\in\RR^n ~\middle|~ f(x) = 0 \right\}
    \end{equation*}
    is nonempty and assume that there exists an open neighborhood $G\subseteq \RR^n$ of $\MM$ such that for all $x\in G$ it holds that $\dim(\ker(Df_x)) = s$, where the constant $s$ is independent of $x$ and satisfies $0 < s < n$.
    Then
    \begin{enumerate}[(i)]
        \item the set $\mathcal{M}$ is an $s$-dimensional $C^2$ submanifold of $\RR^n$,
        \item the tangent space at $p \in\MM$ is given by $T_p\MM = \ker(Df_p)$,
        \item the manifold $\mathcal{M}$ is a $C^{2,1}$ manifold, and
        \item it holds that $\MM \subseteq \mathcal{E}(\MM)$, and if $x\in\mathcal{E}(\MM)\setminus\MM$ then $x-p_x \perp T_{p_x}\MM$.
    \end{enumerate}
\end{proposition}
\begin{proof} \hphantom{a}
\begin{enumerate}[(i)]
    \item
    Let $p\in \MM$. 
    It can be shown~\cite{Wal09} that there exist open neighborhoods $U_1\subseteq \RR^n$ of $p$ and $U_2\subseteq \RR^n$ of $f(p)$ and $C^2$ diffeomorphisms
    \begin{equation*}
        \phi: U_1 \to \phi(U_1)~~\und~~ \psi: U_2 \to \psi(U_2)
    \end{equation*}
    such that
    \begin{equation*}
        f(U_1) \subseteq U_2
    \end{equation*}
    and
    \begin{equation*}
        \psi \circ f \circ \phi\inv (y_{1},\dots,y_n) = (y_1,\dots,y_{n-s}, 0,\dots,0)
    \end{equation*}
    for all $y\in \phi(U_1)$.
    Thus, for each $v \in U_1$ we have
    \begin{align*}
        v \in \MM
    & \Longleftrightarrow  f(v) = 0 \\
    & \Longleftrightarrow  (\psi\circ f)(v) = 0 \\
    & \Longleftrightarrow  (\psi\circ f \circ \phi\inv \circ \phi)(v) = 0 \\
    & \Longleftrightarrow  \phi_1(v) = \dots = \phi_{n-s}(v) = 0.
    \end{align*}

    \item
    Suppose that $p\in \MM$ and that $\phi: U \to \RR^n$ is a local coordinate system of $\MM$ at $p$.
    Since
    \begin{equation*}
        (f\circ\phi\inv) \big(\{ y\in\RR^n ~|~ y_{s+1} = \dots = y_n = 0 \} \cap \phi(U)\big) = \{0\},
    \end{equation*}
    we have
    \begin{equation*}
        Df_p(T_p\MM) = Df_p\Big(D\phi_{\phi(p)}\inv\big(\left\{ y\in\RR^n ~\middle|~ y_{s+1} = \dots = y_n = 0 \right\}\big)\Big) = 0.
    \end{equation*}
    The claim then follows from the fact that $\ker(Df_p)$ has dimension $s$.

    \item
    In~\cite[Equations (3.3) and (3.6)]{DudH94}, it is shown that a $C^k$ manifold with $k\geq 2$ is also a $C^{k,1}$ manifold, from which the claim follows.
    
    \item
    This claim was proven in~\cite{LeoS21}.\qedhere
\end{enumerate}
\end{proof}

We finish this section with \Cref{lem:dist} and a corresponding remark.
The lemma establishes an upper bound on the distance to the manifold $\MM$ defined in \Cref{prop:manifolds} in terms of the function $f$.
This result will be the key in our application to \portHamiltonian systems, as it will allow us to deduce a dissipativity property for~\eqref{ocp}.

\begin{lemma}\label{lem:dist}
    Suppose $f:\RR^n \to \RR^n$ is of class $C^2$, assume that
    \begin{equation*}
        \MM := \left\{ x\in\RR^n ~\middle|~ f(x) = 0 \right\}
    \end{equation*}
    is nonempty and assume that there exists an open neighborhood $G\subseteq \RR^n$ of $\MM$ such that for all $x\in G$ it holds that $\dim(\ker(Df_x)) = s$, where the constant $s$ is independent of $x$ and satisfies $0 < s < n$.
    Further, assume that for each $x\in G$ the smallest nonzero singular value of $Df_x$ is bounded from below by $\tilde{c} > 0$.
    Then $\MM$ is a $C^{2}$ manifold and there exists an open set $V \subseteq \RR^n$ and a constant $c>0$ with
    $\MM \subseteq V \subseteq \mathcal{E}(\MM)$
    and
    \begin{equation}\label{eq:fdist}
        c \dist(x,\MM) \leq \norm{f(x)}
    \end{equation}
    for all $x\in V$.
\end{lemma}
\begin{proof}
    We will first show that~\eqref{fdist} is true locally.
    Fix a point $p\in\MM$ and notice that due to $f\in C^2$, for all $x \in \RR^n$ we have
    \begin{equation}\label{eq:ftaylor}
        f(x) = f(p) + Df_{p}(x-p) + g_{p}(x-p),
    \end{equation}
    where the remainder $g_p$ satisfies 
    \begin{equation}\label{eq:remainder}
        \lim_{x\to p}\frac{\norm{g_p(x-p)}}{\norm{x-p}} = 0.
    \end{equation}
    To establish~\eqref{fdist} locally, our first goal is to show that there exists an open set $U_p \subseteq \mathcal{E}(\MM)$ such that for all $x\in V_p := U_p \cap N_p\MM$ we have 
    \begin{equation}\label{eq:gdurchd}
        \frac{\norm{g_{p}(x-p)}}{\norm{Df_{p}(x-p)}} < \frac12.
    \end{equation}
    For the sake of simplicity, let us set $A_p := Df_{p}\tp Df_{p}$.
    Then, by the definition of $\MM$ and \Cref{prop:manifolds}, we have
    $
        \ker(A_p) = \ker(Df_{p}) = T_{p}\MM.
    $
    Now, let us decompose $\RR^n = \ker(A_p) \oplus \ran(A_p) = T_{p}\MM \oplus N_{p}\MM$ and accordingly also $A_p = 0 \oplus A_p^{(2)}$.
    Since $A_p^{(2)}$ is symmetric positive definite, for any $x\in N_p\MM$ we obtain
    \begin{equation}\label{eq:lambda_min_A2}
        \norm{Df_{p}(x-{p})}^2 = (x-{p})\tp A_p (x-{p}) = (x-{p})\tp A_p^{(2)} (x-{p}) \geq \lambda_{\min}(A_p^{(2)}) \norm{x-{p}}^2.
    \end{equation}
    Now, using~\eqref{remainder} and~\eqref{lambda_min_A2}, we obtain 
    \begin{equation*}
        0 = \lim_{x\to p}\frac{\norm{g_p(x-p)}}{\norm{x-p}} \geq \lim_{x\to p} c_p \frac{\norm{g_p(x-p)}}{\norm{Df_p(x-p)}} \geq 0,
    \end{equation*}
    where we set $c_p:=\lambda_{\min}(A_p^{(2)})^{1/2} > 0$.
    In particular, we have 
    \begin{equation*}
        \lim_{x \to p} \frac{\norm{g_p(x-p)}}{\norm{Df_p(x-p)}} = 0.
    \end{equation*}
    Thus, choosing $x \in N_p\MM$ sufficiently close to $p$ we obtain the estimate~\eqref{gdurchd}, and we can deduce that an open set $U_p$ with the sought-after properties has to exist.
    Now, since~\eqref{ftaylor} and $p\in \MM$ implies
    \begin{equation*}
        \norm{Df_{p}(x-{p})} \leq \norm{f(x)} + \norm{g_{p}(x-{p})}\quad \text{for all} ~~ x \in \RR^n,
    \end{equation*}
    using~\eqref{gdurchd} we obtain
    \begin{equation}\label{eq:12absch}
        \tfrac12 \norm{Df_{p}(x-p)} \leq \norm{f(x)} \quad \text{for all} ~~ x \in V_p = U_p \cap N_p\MM.
    \end{equation}
    To finish the local argument, notice that by~\eqref{lambda_min_A2} we have 
    \begin{equation*}
        c_p \dist(x,\MM) = c_p \norm{x-p} \leq \norm{Df_p(x-p)} \quad \text{for all} ~~ x \in V_p,
    \end{equation*}
    which together with~\eqref{12absch} shows that~\eqref{fdist} holds for all $x\in V_p$.

    To construct the set $V$, first notice that the differentiability of $f$ implies that an expression of the form~\eqref{ftaylor} is possible on the set $\mathcal{E}(\MM)$.
    In other words, there exists a function~$g_{\cdot}(\cdot)$ such that
    \begin{equation*}
        f(x) = f(p_x) + Df_{p_x}(x-p_x) + g_{p_x}(x-p_x)
    \end{equation*}
    for all $x\in\mathcal{E}(\MM)$.
    Since the orthogonal projection $x\mapsto p_x$ is differentiable~\cite{LeoS21} and $f\in C^2$, the map $x \mapsto g_{p_x}(x-p_x)$ is continuous on $\mathcal{E}(\MM)$.
    Define the function
    \begin{equation*}
        h: \mathcal{E}(\MM) \to \RR, ~ x \mapsto \frac{\norm{g_{p_x}(x-p_x)}}{\norm{Df_{p_x}(x-p_x)}}.
    \end{equation*}
    Then $h$ is continuous and hence the preimage of the open set $(-\infty,\tfrac12)\subseteq \RR$ under $h$ is open.
    Note that $\MM$ is a subset of this preimage.
    Define
    \begin{equation*}
        \MM \subseteq V := h\inv\big((-\infty,\tfrac12)\big) \subseteq \mathcal{E}(\MM).
    \end{equation*}
    Then for each $p\in \MM$ we have
    \begin{equation*}
        V_p \subseteq V \cap N_p\MM.
    \end{equation*}
    The previous arguments can then be used to show that for $c := \tfrac{\tilde{c}}{2}$ and $x \in V$ we have
    \begin{equation*}
        2c \dist(x,\MM) \leq c_{p_x} \dist(x,\MM) \leq \norm{Df_{p_x}(x-{p_x})} \leq 2 \norm{f(x)},
    \end{equation*}
    finishing the proof.
\end{proof}

\begin{remark}\label{rem:ineq}
    The estimate~\eqref{fdist} is related to the {\L}ojasiewicz inequality~\cite{Loj59,JiKS92}, which states that for a real analytic function $g: U \to \RR$ defined on an open set $U\subseteq \RR^n$ and a compact set $K \subseteq U$, the distance of $x \in K$ to the zero locus $\mathcal{Z} := \{ z \in U ~|~ g(z) = 0 \}$ of $g$ may be estimated by 
    \begin{equation*}
        \dist(x, \mathcal{Z})^\alpha \leq C~|g(x)|,
    \end{equation*}
    where $\alpha$ and $C$ are positive constants.
\end{remark}

\section{Manifold dissipativity and manifold turnpikes}
\label{sec:dissip}

In this section, we recall the definition of dissipativity with respect to a manifold and the definition of manifold turnpikes as introduced in~\cite{FauFO22}.
Further, a theorem relating the two properties is stated.
Here and in the following, the set $\mathcal{K}$ is defined as
\begin{equation*}
    \mathcal{K} := \{ \alpha: [0,\infty) \to [0,\infty) ~|~ \alpha(0)=0, ~ \alpha \text{ is continuous and strictly increasing}\}.
\end{equation*}

We consider the general optimal control problem
\begin{equation}\label{eq:gen-ocp}
    \left.
    \begin{aligned}
        \min_{u \in \mathcal{U}}~C_T(u) := & \int_{0}^{T} \ell(x,u) \dt \\
        \text{subject to} \qquad & \\
        h(x) \dot{x} = &~ g(x,u), \\
        x(0) = x_0, \quad x(T) = x_T.\hspace{-1.95cm}
    \end{aligned}
    \tag{$\text{OCP}_T$}
    \right\}
\end{equation}
As before, we assume $x_0, x_T \in K$, where $K \subseteq \RR^n$ is a compact set.
Here, the function $g$ defines the dynamics of the system and the function $h$ corresponds to possible algebraic constraints.
We refrain from further specification of these functions as~\eqref{gen-ocp} is only used for general definitions.
Throughout this section, we assume that an optimal control~$u^*$ of~\eqref{gen-ocp} and an associated trajectory $x^*$ exist.

We begin with the definition of manifold dissipativity.
The definition is related to Willems' notion of dissipativity~\cite{Wil72a} and is also found in~\cite{FauFO22}.

\begin{definition}[manifold dissipativity]\label{def:dissip}
    Consider the optimal control problem~\eqref{gen-ocp} together with the manifold $\MM\subseteq\RR^n$.
    We say that~\eqref{gen-ocp} is \emph{dissipative with respect to the manifold $\MM$} if there exists a function $\mathcal{S}: \RR^n \to [0,\infty)$ that is bounded on compact sets and a function $\alpha\in\mathcal{K}$ such that all optimal controls $u^*$ and associated trajectories $x^*$ satisfy the dissipation inequality
    \begin{equation}\label{eq:gen-dissip}
        \mathcal{S}(x_T) - \mathcal{S}(x_0) \leq \int_{0}^{T} \ell(x^*,u^*) - \alpha(\dist(x^*,\MM)) \dt
    \end{equation}
    for all $T >0$.
\end{definition}

The function $\mathcal{S}$ from \Cref{def:dissip} is also called storage function.
Note that we require the dissipation inequality~\eqref{gen-dissip} only to hold along optimal solutions of~\eqref{gen-ocp}.
This is not a severe restriction when turnpike phenomona are studied, as we are only interested in properties of optimal solutions.

Next, we define a manifold turnpike property, again following~\cite{FauFO22}.
The property is essentially a notion of measure turnpikes, see, e.g.,~\cite{CarHL91,TreZ18,FauKJB17,Zas06}.

\begin{definition}[manifold turnpike]
    Consider the optimal control problem~\eqref{gen-ocp}
    together with the manifold $\MM\subseteq\RR^n$.
    We say that~\eqref{gen-ocp} has the \emph{manifold turnpike property with respect to the manifold $\MM$} if for all compact sets $K\subseteq \RR^n$ and all $\varepsilon>0$ there exists a constant $C_{K,\varepsilon} > 0$ such that for all $T>0$ all optimal trajectories $x^*$ of~\eqref{gen-ocp} satisfy
    \begin{equation*}
        \lambda\big(\{ t\in[0,T] ~|~ \dist(x^*(t),\MM) > \varepsilon\}\big) \leq C_{K,\varepsilon}
    \end{equation*}
    for all $x_0,x_T \in K$.
    Here, $\lambda$ denotes the Lebesgue-measure.
\end{definition}

The next theorem can be found similarly in~\cite{FauFO22,FauKJB17,GruG21}.
The theorem shows that manifold dissipativity implies a manifold turnpike property.

\begin{theorem}[manifold dissipativity implies manifold turnpike]\label{thm:diss->turn}
    Consider the optimal control problem~\eqref{gen-ocp} together with a submanifold $\MM \subseteq \RR^n$ and assume that
    \begin{enumerate}[(i)]
        \item \label{as:i-diss-turn} there exists a constant $C_\ell(K)>0$ such that for all optimal controls $u^*$ of~\eqref{gen-ocp} and the associated trajectories $x^*$ we have
        \begin{equation*}
            \int_0^T \ell(x^*,u^*) \dt < C_\ell(K)
        \end{equation*}
        for all $T > 0$, and
        \item the optimal control problem is dissipative with respect to the manifold $\MM$.
    \end{enumerate}
    Then the optimal control problem~\eqref{gen-ocp} has the manifold turnpike property.
\end{theorem}

\begin{remark}\label{rem:cost-controllability}
    In~\cite{FauFO22}, \Cref{thm:diss->turn} is stated with a stronger assumption in the place of~\ref{as:i-diss-turn}, which can be interpreted as a controllability property.
    For the sake of simplicity, we do not consider this case.
\end{remark}

\section{Application to \portHamiltonian systems}\label{sec:app-ph}
Finally, we are ready to apply the previous results to \portHamiltonian systems and the optimal control problem~\eqref{ocp}.
First, recall that we can rewrite the cost functional $C_{\pH,T}(u)$ as
\begin{equation}\label{eq:cu}
    C_{\pH,T}(u) = \int_{0}^{T} y\tp u \dt = \mathcal{H}(x_T) - \mathcal{H}(x_0) + \int_{0}^{T}\norm{R(x)^{1/2} \eta(x)}^2 \dt,
\end{equation}
and that rearranging gives
\begin{equation}\label{eq:cu2}
    \mathcal{H}(x_T) - \mathcal{H}(x_0) = \int_{0}^{T}y\tp u - \norm{R(x)^{1/2}\eta(x)}^2 \dt.
\end{equation}
Equation~\eqref{cu} hints that any optimal trajectory will have to spend most of the time close to the set
\begin{align*}
    \MM
    := & \big\{ x\in\RR^n ~\big|~ R(x)^{1/2}\eta(x) = 0 \big\},
\end{align*}
and that $\mathcal{H}$ can be used as a storage function to derive dissipativity notions with respect to~$\MM$.
Our aim will be to formalize these ideas.
The first step will be to ensure that~$\MM$ has the necessary manifold structure.
For that, we make the following assumptions.

\begin{assumption}\label{as:ph}\hphantom{abc}
    \begin{enumerate}[({A}1)]
        \item \label{as:A1} The map $f: \RR^n \to \RR^n, ~ x\mapsto R(x)^{1/2} \eta(x)$ is of class $C^2$.
        \item \label{as:A2} The set $\MM$ is nonempty and there exists an open neighborhood $G\subseteq \RR^n$ of $\MM$ such that for all $x\in G$ it holds that $\dim(\ker(Df_x)) = s$, where the constant $s$ is independent of $x$ and satisfies $0 < s < n$.
        \item \label{as:A3} For each $x\in G$, the smallest nonzero singular value of $Df_x$ is bounded from below by a positive constant $\tilde{c} > 0$.
        \item \label{as:A4} Let $V$ be the open set from \Cref{lem:dist}. Any optimal trajectory $x^*$ of~\eqref{ocp} remains in $V$ for all times.
    \end{enumerate}
\end{assumption}

With assumptions \ref{as:A1} and \ref{as:A2}, \Cref{prop:manifolds} ensures that $\MM$ is an $s$-dimensional $C^{2,1}$ submanifold of $\RR^n$ and that $\MM \subseteq \mathcal{E}(\MM)$.

The next step is to show that the problem is dissipative.
As we will see shortly, this is ensured by assumptions \ref{as:A3} and \ref{as:A4}, which allow us to use \Cref{lem:dist} to conclude that the optimal control problem~\eqref{ocp} is dissipative with respect to the manifold~$\MM$.
Notice that \ref{as:A4} implies $x_0,x_T \in V$.

\begin{theorem}[\eqref{ocp} is dissipative]\label{thm:ph-dissip}
    Under \Cref{as:ph}, the optimal control problem~\eqref{ocp} is dissipative with respect to the manifold $\MM$ with storage function $\mathcal{H}$.
\end{theorem}
\begin{proof}
    The proof is essentially an application of \Cref{lem:dist}.
    Since all assumptions of \Cref{lem:dist} are satisfied under \Cref{as:ph}, there exists an open set $V\subseteq \RR^n$ and constant $c>0$ such that $\MM \subseteq V \subseteq \mathcal{E}(\MM)$ and
    \begin{equation}\label{eq:fdist-ph}
        c \dist(x,\MM) \leq \norm{f(x)} = \norm{R(x)^{1/2} \eta(x)}
    \end{equation}
    holds for all $x\in V$.
    In particular, assumption \ref{as:A4} ensures that the estimate holds along any optimal trajectory $x^*$ of~\eqref{ocp}.
    With this and~\eqref{cu2}, we see that for any optimal control $u^*$ and the associated trajectory $x^*$ and output $y^*$, we have
    \begin{align*}
        \mathcal{H}(x_T) - \mathcal{H}(x_0)
    & =    \int_{0}^{T}{y^*}\tp u^* - \norm{R(x^*)^{1/2}\eta(x^*)}^2 \dt \\
    & \leq \int_{0}^{T}{y^*}\tp u^* - c^2 \dist(x^*,\MM)^2 \dt \\
    & =    \int_{0}^{T}{y^*}\tp u^* - \alpha(\dist(x^*,\MM)) \dt,
    \end{align*}
    where $\alpha: s \mapsto c^2 s^2 \in \mathcal{K}$.
    Finally, note that the Hamiltonian $\mathcal{H}$, acting as a storage function here, is bounded on compact sets since it is differentiable.
\end{proof}

\begin{remark}\label{rem:ineq2}
    Let us emphasize that the estimate~\eqref{fdist} from \Cref{lem:dist} was the key to conclude the dissipativity property of~\eqref{ocp}.
    As we have mentioned in \Cref{rem:ineq}, the estimate is related to {\L}ojasiewicz' inequality~\cite{Loj59,JiKS92}.
    In fact, if the map $g: x \mapsto \norm{f(x)}^2$ is real analytic, then we may use the {\L}ojasiewicz inequality to derive dissipativity without~\Cref{lem:dist}, as long as any optimal trajectory stays in some compact set $D\subseteq \RR^n$.
\end{remark}

Now, an application of \Cref{thm:diss->turn} yields the following result, showing that the optimal control problem~\eqref{ocp} has the manifold turnpike property with respect to~$\MM$.

\vbox{
\begin{theorem}[\eqref{ocp} has a manifold turnpike]
    \label{thm:ph-turn}
    In addition to \Cref{as:ph}, assume
    \begin{enumerate}[({A}1)]
        \setcounter{enumi}{4}
        \item \label{as:mta1}there exists a control $u_1 \in \Uad$ that steers the associated trajectory $x_1$ from $x_0$ onto the manifold~$\MM$ in time $T_1 \geq 0$, and
        \item \label{as:mta2}there exists a control $u_2 \in \Uad$ that steers the associated trajectory $x_2$ from the manifold $\MM$ to $x_T$ in time $T_2 \geq 0$.
    \end{enumerate}
    Then the optimal control problem~\eqref{ocp} has a manifold turnpike at the manifold $\MM$.
\end{theorem}
}\vspace{-0.4cm}
\begin{proof}
    Since the \portHamiltonian system~\eqref{nlph} is autonomous and there is no control cost on the manifold $\MM$, the conditions \ref{as:mta1} and \ref{as:mta2} ensure that the total cost of the optimal control $u^*$ is bounded by
    \begin{align*}
        C_{\pH,T}(u^*)
    & = \int_{0}^{T}{y^*}\tp u^* \dt \\
    & = \mathcal{H}(x_T) - \mathcal{H}(x_0) + \int_{0}^{T}\norm{R(x^*)^{1/2} \eta(x^*)}^2 \dt \\
    & \leq C_{\mathcal{H}}(K) + \int_{0}^{T_1} \norm{R(x_1)^{1/2} \eta(x_1)}^2 \dt + \int_{0}^{T_2} \norm{R(x_2)^{1/2} \eta(x_2)}^2 \dt \\
    & \leq C_{\mathcal{H}}(K) + C_1(K) + C_2(K) < \infty.
    \end{align*}
    Notice that the constants $C_{\mathcal{H}}(K), ~C_1(K)$ and $C_2(K)$ are independent of the final time $T$.
    Thus, using the results of \Cref{thm:ph-dissip}, we may apply \Cref{thm:diss->turn} to conclude that the optimal control problem~\eqref{ocp} has the manifold turnpike property at $\MM$.
\end{proof}

In \Cref{thm:ph-turn}, in order to show that a turnpike property holds true, we needed to make the controllability assumptions~\ref{as:mta1} and~\ref{as:mta2}.
This is a common pattern in turnpike results, similar controllability assumptions are made in~\cite{CarHL91,FauFO22} and~\cite{SchPFWM21}.

\begin{remark}
    In~\cite{SchPFWM21}, the authors considered linear \pH systems of the form 
    \begin{align*}
        \dot{x} & = (J-R)Qx + Bu, \\
              y & = B\tp Q x,
    \end{align*}
    where $J=-J\tp$, $R=R\tp \succeq 0$ and $Q=Q\tp \succ 0$.
    They have shown that in this case the optimal control problem~\eqref{ocp} admits a subspace turnpike property with respect to $\ker(R^{1/2} Q)$.
    We can interpret \Cref{thm:ph-turn} as a generalization of this result to the nonlinear case.
    In the linear case, assumption \ref{as:A1} is immediately satisfied.
    Further, the set $\MM$ is the kernel of $R^{1/2} Q$, and if the dimension of $\ker(R^{1/2} Q)$ is $0<s<n$, then assumptions \ref{as:A2} and \ref{as:A3} are also satisfied.
    Since the distance estimate~\eqref{fdist-ph} can be shown to hold true globally~\cite[Lemma 13]{SchPFWM21}, the set $V$ from \Cref{lem:dist} is $V=\RR^n$ and thus assumption \ref{as:A4} is also satisfied.
\end{remark}

We finish this section with a simple example illustrating \Cref{thm:ph-turn}.

\begin{example}\label{ex:ph}
    Consider the functions $E,J,R,\eta$ and $B$ defined by
    \begin{gather*}
        E(x)   = \vec{ 1 & 0 \\ 0 & 1}, ~~
        J(x)   = \vec{ 0 & 1 \\ -1 & 0 }, ~~
        B(x)   = \vec{ 1 \\ 0 }, \\
        R(x)   = \vec{ \tfrac14 (4\norm{x}^2 + 1)^2 & 0 \\ 0 & 0}, ~~
        \eta(x)  = \vec{2 & 0 \\ 0 & 1}x
    \end{gather*}
    for all $x\in \RR^2$, which together with the Hamiltonian
    \begin{equation*}
        \mathcal{H}(x) = \frac12 ~x\tp\!\vec{2 & 0 \\ 0 & 1} x
    \end{equation*}
    form the \portHamiltonian system 
    \begin{equation}\label{eq:ph-ex-ode}
        \begin{aligned}
            E(x) \dot{x} = &~ \big(J(x) - R(x)\big)\eta(x) + B(x) u, \\
            y = &~ B(x)\tp \eta(x).
        \end{aligned}
        \tag{\textsf{pH-1}}
    \end{equation}
    For the system~\eqref{ph-ex-ode}, the function $f$ reads as 
    \begin{equation*}
        f: \RR^2 \to \RR^2,~ x \mapsto R(x)^{1/2} \eta(x) = \vec{4 \norm{x}^2 + 1 & 0 \\ 0 & 0}x = \vec{4(x_1^3 + x_2^2x_1) + x_1 \\ 0},
    \end{equation*}
    where we take $x=\vec{x_1 & x_2}\tp$.
    Thus, the derivative $Df$ reads as 
    \begin{equation*}
        Df(x) = \vec{ 12 x_1^2+ 4 x_2^2 + 1 & 8 x_2 x_1 \\ 0 & 0} \neq 0
    \end{equation*}
    and the subspace $\ker(Df(x))$ is one-dimensional for all $x\in\RR^2$.
    A simple calculation shows that the nonzero singular value of $Df(x)$ is given by 
    \begin{equation*}
        \sigma(x) = 144 x_1^4 + 160 x_1^2 x_2^2 + 24 x_1^2 + 16 x_2^4 + 8 x_2^2 + 1 \geq 1.
    \end{equation*}
    The zero locus $\MM$ of $f$ is given by 
    \begin{equation*}
        \MM = \bigg\{ \vec{x_1 \\ x_2} \in \RR^2 ~\bigg|~ (4 x_1^2 + 4x_2^2 + 1) x_1 = 0\bigg\} = \bigg\{ \vec{x_1 \\ x_2} \in \RR^2 ~\bigg|~ x_1 = 0\bigg\}.
    \end{equation*}
    Hence, assumptions \ref{as:A1}, \ref{as:A2} and \ref{as:A3} are satisfied for~\eqref{ph-ex-ode}.
    As $\MM$ is a linear subspace, the orthogonal projection on $\MM$ is well defined globally and we have $\mathcal{E}(\MM)=\RR^2$.
    Further, for the set $V$ from \Cref{lem:dist} it holds that $V=\RR^2$ since
    \begin{equation*}
        \dist(x,\MM) = |x_1| \leq | x_1 (4x_1^2 + 4x_2^2 + 1)| = \norm{f(x)}
    \end{equation*}
    for all $x\in \RR^2$.
    Thus, also assumption \ref{as:A4} is satisfied for~\eqref{ph-ex-ode}.

    Now, for $\xi = \vec{\xi_1 & \xi_2 & \xi_3}\tp \in \RR^3$, let us define the functions $\tilde{E},\tilde{J},\tilde{R},\tilde{\eta}$ and $\tilde{B}$ by
    \begin{gather*}
        \tilde{E}(\xi) := \vec{1 & 0 & 0 \\ 0 & 0 & 1 \\ 0 & 0 & 0}, ~~
        \tilde{J}(\xi) := \vec{0 & 1 & 0 \\ -1 & 0 & 0 \\ 0 & 0 & 0}, ~~
        \tilde{B}(\xi) := \vec{1 \\ 2 \\ 0}, \\
        \tilde{R}(\xi) := \vec{ \tfrac14 (4\xi_1^2 + 4\xi_2^2 + 1)^2 & 0 & 0 \\ 0 & 0 & 0 \\ 0 & 0 & 0},
        ~~
        \tilde{\eta}(\xi) := \vec{2\xi_1 \\ \xi_2 \\ \xi_3}.
    \end{gather*}
    It is easy to see that also the system 
    \begin{equation}\label{eq:ph-ex-dae}
        \begin{aligned}
            \tilde{E}(\xi) \dot{\xi} = &~ \big(\tilde{J}(\xi) - \tilde{R}(\xi)\big)\tilde{\eta}(\xi) + \tilde{B}(\xi) \tilde{u}, \\
            \tilde{y} = &~ \tilde{B}(\xi)\tp \tilde{\eta}(\xi)
        \end{aligned}
        \tag{\textsf{pH-2}}
    \end{equation}
    satisfies assumptions \ref{as:A1} -- \ref{as:A4}.
    The zero locus of the map $\tilde{f}: \xi \mapsto \tilde{R}(\xi)^{1/2}\tilde{\eta}(\xi)$ is 
    \begin{equation*}
        \tilde{\MM} = \left\{ \vec{\xi_1 & \xi_2 & \xi_3}\tp \in \RR^3 ~\middle|~ \xi_1 = 0 \right\}.
    \end{equation*}
\end{example}

\section{Numerical example}\label{sec:num}

As an example, we consider the optimal control problem~\eqref{ocp} together with the \pH systems~\eqref{ph-ex-ode} and~\eqref{ph-ex-dae} from \Cref{ex:ph}.
For the implementation, we use the open-source software package \texttt{CasADi}~\cite{AndGH18}.

In order to use \texttt{CasADi}, the optimal control problem~\eqref{ocp} is formulated as a minimization problem of the form 
\begin{equation}\label{eq:cs-nlp}
\begin{gathered}
    \min_{w} ~J(w) \\
    \text{subject to }~ w_{\text{lb}} \leq w \leq w_{\text{ub}} \und G(w) = 0.
\end{gathered}
\end{equation}
We follow a similar procedure as~\cite[Section 5.4]{AndGH18}.
In our implementation, $w$ contains the values $x(t_i)$ for the discretization points $t_i \in [0,T]$, and the values $u(t_i)$ for the discretization points $t_i \in [0,T], ~ i \neq 0$.
The initial condition and possible control constraints are incorporated in $w_{\text{lb}}$ and $w_{\text{ub}}$.
The function $G$ is used to enforce the final condition and a continuity condition on $x$ by using an integrator scheme to determine the value $x(t_{i+1})$ given the values $x(t_i)$ and $u(t_i)$.
This integrator scheme is also used to calculate the cost~$J$ via the \texttt{quad} option in \texttt{CasADi}'s \texttt{integrator} function.
For the solution of the nonlinear optimization problem~\eqref{cs-nlp}, \textsf{Ipopt}~\cite{WaeB06} is used.

In \Cref{fig:ocp-ex-ode}, the solution of the optimal control problem~\eqref{ocp} with the system~\eqref{ph-ex-ode} under the control constraint $-50 \leq u(t) \leq 50$ is shown. 
The turnpike behaviour is clearly visible; the first component $x_1$ of the optimal trajectory $x^*$ approaches the manifold $\MM = \{ x \in \RR^2 ~|~ x_1 = 0 \}$ very quickly and remains there for the majority of the time horizon.
The same observation can be made for larger time horizons, which is not shown in \Cref{fig:ocp-ex-ode}.

In \Cref{fig:ocp-ex-dae}, the solution of the optimal control problem~\eqref{ocp} with the system~\eqref{ph-ex-dae} under the control constraint $-200 \leq u(t) \leq 200$ is shown. 
Again, the turnpike phenomenon can be observed.

\begin{figure}[ht]
    \centering
    \input{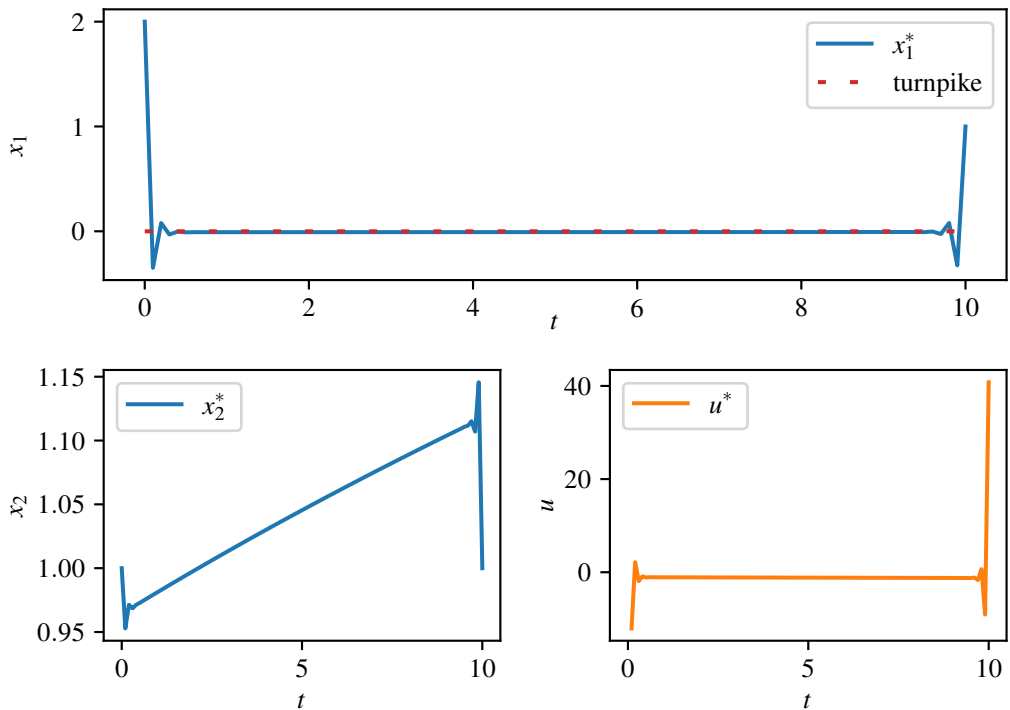}
    \vspace{-.7cm}
    \caption{Minimal energy control for~\eqref{ph-ex-ode}.
    As initial and final values $x_0 = [\begin{smallmatrix} 2 \\ 1 \end{smallmatrix}]$ and $x_T = [\begin{smallmatrix} 1 \\ 1 \end{smallmatrix}]$ are chosen.
    The considered time horizon is $[0,T]$ with final time $T = 10$, and the control is constrained via $-50\leq u(t) \leq 50$.
    We used $100$ discretization steps.}
    \label{fig:ocp-ex-ode}
\end{figure}

\begin{figure}[ht]
    \centering
    \input{./figures/example-dae.pgf}
    \vspace{-.7cm}
    \caption{Minimal energy control for~\eqref{ph-ex-dae}.
    As initial and final values $x_0 = [\begin{smallmatrix} 2 & 1 & 10 \end{smallmatrix}]\tp$ and $x_T = [\begin{smallmatrix} 1 & 2 & 20 \end{smallmatrix}]\tp$ are chosen.
    The considered time horizon is $[0,T]$ with final time \hbox{$T = 20$}, and the control is constrained via $-200\leq u(t) \leq 200$.
    We used $400$ discretization steps.}
    \label{fig:ocp-ex-dae}
\end{figure}

\section{Conclusion}\label{sec:conclusion}
In this paper, we have considered the optimal control of \portHamiltonian systems under minimal energy supply with fixed initial and final values.
We have seen that a map $f$, corresponding to the energy dissipating portion of the right hand side, and its zero locus $\MM = \{x ~|~ f(x) = 0 \}$, which corresponds to the dissipative part of the state space, play an important role.
It was shown that under smoothness assumptions on $f$, the set $\MM$ forms a $C^2$ submanifold of $\RR^n$.
In particular, using results from~\cite{LeoS21}, we observed that the orthogonal projection onto $\MM$ is well-defined in an open set $\mathcal{E}(\MM)$.
Further, we have shown that under these assumptions the distance of a point $x$ to $\MM$ can essentially be bounded by $\norm{f(x)}$ from above.
This fact allowed us to deduce that the considered optimal control problem is dissipative with respect to the manifold~$\MM$.
Our main result was a consequence of this dissipativity property.
Under additional controllability assumptions, we have seen that the problem has a manifold turnpike property with respect to $\MM$.
This theoretical observation was confirmed in a simple numerical example.

An open question from a theoretical perspective is the existence of optimal controls of~\eqref{ocp}.
Here, similar to~\cite{SchPFWM21,FauMP22,PhiSF21}, the particular structure of \pH systems should be exploited.
Another open topic is the study of stronger turnpike properties, such as exponential turnpikes~\cite{TreZZ18,GruG21}.
Applications of the theoretical results to specific \portHamiltonian systems such as gas networks~\cite{DomHL21} will be studied in future works.

\subsection*{Acknowledgements} 
\vbox{
The author thanks Tobias Breiten, Bernhard Höveler and Volker Mehrmann for their helpful comments on an early version of this manuscript.
Further, the author thanks the Deutsche Forschungsgemeinschaft for their support within the subproject B03 in the Son\-der\-for\-schungs\-be\-reich/Trans\-re\-gio 154 ``Mathematical Modelling, Simulation and Optimization using the Example of Gas Networks'' (Project 239904186).}

\bibliographystyle{siam}
\bibliography{../--literature/references}

\end{document}